\documentclass[11pt]{amsart}

\usepackage[T1]{fontenc}
\usepackage{amssymb}
\usepackage[colorlinks=true,linkcolor=blue,citecolor=blue,urlcolor=blue]{hyperref}

\newcommand{\C}{\mathbb C}
\newcommand{\A}{\mathbb A}
\newcommand{\Ga}{\mathbb G_a}
\newcommand{\Aut}{\operatorname{Aut}}
\newcommand{\SAut}{\operatorname{SAut}}
\newcommand{\Ker}{\operatorname{Ker}}
\newcommand{\LND}{\operatorname{LND}}
\newcommand{\ML}{\operatorname{ML}}
\newcommand{\Xcal}{{\mathcal X}}
\newcommand{\Ycal}{{\mathcal Y}}
\newcommand{\Ocal}{{\mathcal O}}
\newcommand{\Spec}{\operatorname{Spec}}
\newcommand{\pr}{\operatorname{pr}}

\newtheorem{theorem}{Theorem}[section]
\newtheorem{proposition}[theorem]{Proposition}
\newtheorem{corollary}[theorem]{Corollary}
\newtheorem{lemma}[theorem]{Lemma}
\theoremstyle{remark}

\theoremstyle{definition}
\newtheorem{construction}[theorem]{Construction}

\title{The cylinder over the Russell cubic threefold is flexible}
\author{Pierre-Marie Poloni}
\address{}
\email{poloni.pierremarie@gmail.com}
\date{\today}


\begin{document}

\begin{abstract}
The Russell cubic threefold is the hypersurface
\(\Xcal\subset\A^4_{\C}\) defined by
\(x^2y+z^2+x+t^3=0\). We prove that its cylinder
\(\Xcal\times\A^1\) is flexible. More precisely, after replacing its
standard embedding in \(\A^5\) by an equivalent one, we construct four
explicit algebraic \(\Ga\)-actions on the ambient affine space whose
restrictions generate a group acting transitively on the cylinder. 
\end{abstract}

\maketitle

\section{Introduction}

Let \(P=x^2y+z^2+x+t^3\in\C[x,y,z,t]\). The Russell cubic
threefold is the smooth affine hypersurface
\(\Xcal=\{P=0\}\subset\A^4\). We begin by recalling some of its
geometry.

Consider the projection \(\pr_x\colon\Xcal\to\A^1\),
\((x,y,z,t)\mapsto x\). Every fiber over a nonzero point is isomorphic
to \(\A^2\), whereas the fiber over the origin is the ``book
surface''
\(\Xcal\cap\{x=0\}\simeq\Gamma_{2,3}\times\A^1\), where
\(\Gamma_{2,3}\) is the cuspidal curve defined by \(z^2+t^3=0\).
The projection \((x,y,z,t)\mapsto(x,z,t)\) realizes \(\Xcal\) as an
affine modification of \(\A^3\): the plane \(\{x=0\}\) is replaced by
the exceptional divisor \(\Gamma_{2,3}\times\A^1\). It follows from
\cite[Theorem~3.1]{KZ99} that \(\Xcal\) is contractible. The
Dimca--Ramanujam theorem then implies that its underlying real manifold
is diffeomorphic to \(\mathbb R^6\).

Nevertheless, \(\Xcal\) is not algebraically isomorphic to \(\A^3\).
Indeed, its Makar-Limanov invariant is
\(\ML(\Xcal)=\C[x]\), which is nontrivial \cite{ML96}. In particular, every algebraic
\(\Ga\)-action on \(\Xcal\) preserves the regular function \(x\).
Thus, compared with affine three-space, the Russell
cubic does not admit sufficiently many algebraic additive group
actions.

The automorphism group of \(\Xcal\) was determined in
\cite{DMJP10}.  The resulting orbit decomposition, recorded explicitly in \cite[Proposition~1]{DMJP14}, is as follows.

\begin{theorem}[{\cite[Proposition~1]{DMJP14}}]
\label{thm:Russell-orbits}
The group \(\Aut(\Xcal)\) acts on \(\Xcal\) with exactly four orbits:
\begin{itemize}
\item the open orbit
\(\Ocal_1=\{x\ne0\}\simeq\A^1_*\times\A^2\);
\item the orbit
\(\Ocal_2=\{x=0,\ (z,t)\ne(0,0)\}
\simeq\A^1_*\times\A^1\);
\item the punctured line
\(\Ocal_3=\{(0,y,0,0):y\in\C^*\}\simeq\A^1_*\);
\item the fixed point \(\Ocal_4=\{(0,0,0,0)\}\).
\end{itemize}
\end{theorem}

We now turn to the cylinder
\(\Xcal\times\A^1=\Spec(\C[x,y,z,t,u]/(P))\). In contrast with the
Russell cubic itself, Dubouloz proved that
\(\ML(\Xcal\times\A^1)=\C\) \cite{Dub09}. The rigidity detected by
the Makar-Limanov invariant therefore disappears after passing to the
cylinder. However, this does not by itself determine how large the automorphism
group of the cylinder is, and in particular does not imply
transitivity.

Together with translations along the $\A^1$-factor,
Theorem~\ref{thm:Russell-orbits} immediately shows that
$\Aut(\Xcal\times\A^1)$ has at most four orbits. In fact, the results of \cite{DMJP10} already allow us to reduce
this upper bound to two. Since this sharper conclusion is not stated explicitly
there, we briefly give the argument. First, an explicit locally nilpotent derivation
\(D\in\LND(\C[\Xcal\times\A^1])\) satisfying \(D(x)\ne0\) is
constructed in \cite[Section~7]{DMJP10}. The corresponding \(\Ga\)-action can be
used to connect
\(\Ocal_1\times\A^1\) with \(\Ocal_2\times\A^1\). Second, one can connect
\(\Ocal_3\times\A^1\) with \(\Ocal_4\times\A^1\) by using another embedding of $\Xcal$, inequivalent to the standard embedding but stably equivalent to it. Let
\(Q=x^2y+(1+x)(z^2+x+t^3)\), and let
\(\Xcal'=\{Q=0\}\subset\A^4\). The morphism
\(\varphi\colon\Xcal\to\Xcal'\),
\((x,y,z,t)\mapsto(x,(1+x)y,z,t)\), is an isomorphism fixing the
origin. On the other hand, the proof of
\cite[Proposition~6.4]{DMJP10} provides an automorphism of \(\A^5\)
which maps \(\Xcal'\times\A^1\) onto \(\Xcal\times\A^1\) and sends
\((0,0,0,0,0)\) to \((0,1,0,0,0)\). Composing its restriction with
\(\varphi\times\operatorname{id}_{\A^1}\) gives an automorphism of
\(\Xcal\times\A^1\) which connects
\(\Ocal_3\times\A^1\) and \(\Ocal_4\times\A^1\).

Let \(L=\{x=z=t=0\}=\Ocal_3\sqcup\Ocal_4\subset\Xcal\) be the distinguished line inside the Russell cubic. The
preceding discussion shows that, in order
to prove transitivity on $\Xcal\times\A^1$,  it is enough to construct an
automorphism of the cylinder moving a point out of
\(L\times\A^1\). The required automorphism will arise from a new locally nilpotent
derivation \(E\).

Our starting point is a slight normalization of the locally nilpotent
derivation \(D\) introduced in \cite{DMJP10}. We then apply to it
a procedure referred to as Bhatwadekar's technique in
\cite[Section~7]{GP26}, localizing at a suitable element of
\(\Ker(D)\) that restricts to the constant function \(1\) on \(L\times\A^1\).

It turns out that the explicit formulas for both \(D\) and \(E\)
become considerably simpler if we replace the standard embedding of
the cylinder over the Russell cubic by the following equivalent one. Consider the automorphism \(
\Theta(x,y,z,t,u)=(x,y+u^2,z-xu,t,u)
\) of \(\A^5\), and put \(\widetilde P=x^2y+2xzu+z^2+x+t^3\). Then \(\Theta^*(\widetilde P)=P\). Thus \(\Theta\) restricts to an
isomorphism \(
\Theta\colon\Xcal\times\A^1\xrightarrow{\sim}\Ycal\), where \(\Ycal=V(\widetilde P)\subset\A^5\).

From this point onward, all explicit formulas are written in these
adapted coordinates. The new locally nilpotent derivation is described in the following
proposition.

\begin{proposition}
\label{prop:new-LND}

Set \(a=1+xy+2zu\), \(c=zy-2au\), and \(m=zy+2au\). Consider the
derivation \(E\) of \(\C[x,y,z,t,u]\) defined by
\[
\begin{aligned}
E(x)&=3t^2m, & E(y)&=0, & E(z)&=-3ayt^2,\\
E(t)&=ac,      & E(u)&=-\frac32y^2t^2.
\end{aligned}
\]
Then
\[
E(a)=E(c)=0
\qquad\text{and}\qquad
E(m)=-6ay^2t^2.
\]
Moreover, \(E\) is locally nilpotent and annihilates
\(E(\widetilde P)=0\). It therefore induces a locally nilpotent derivation
of \(\C[\Ycal]\).
\end{proposition}

To illustrate the geometric effect of \(E\), consider the point
\(p=(0,1,0,0,1)\) in \(\Ycal\). Its inverse image under \(\Theta\) is
\((0,0,0,0,1)\in\Ocal_4\times\A^1\). Moreover,
\(a(p)=1\) and \(c(p)=-2\), so \(E(t)(p)=-2\). Consequently, the \(\Ga\)-orbit of
\(p\) associated with \(E\) is not contained in the plane
\(\{x=z=t=0\}\subset\Ycal\), which is the image of
\(L\times\A^1\) under \(\Theta\).

Combined with the preceding discussion, this shows that the two
remaining possible orbits meet. Hence \(\Aut(\Ycal)\) acts transitively
on \(\Ycal\), and, via the isomorphism
\(\Theta\colon\Xcal\times\A^1\xrightarrow{\sim}\Ycal\), we get the following result.

\begin{corollary}\label{cor:cylinder-transitivity}
The automorphism group $\Aut(\Xcal\times\A^1)$ acts transitively on
$\Xcal\times\A^1$.
\end{corollary}

In fact, we can strengthen this conclusion by showing that transitivity already
holds for the special automorphism group. Recall that, for an affine
variety \(V\), the group \(\SAut(V)\) is the subgroup of
\(\Aut(V)\) generated by all one-parameter unipotent subgroups, or
equivalently by the automorphisms arising from algebraic
\(\Ga\)-actions on \(V\). Moreover, by \cite[Theorem~0.1]{AFKKZ13}, the transitivity
of \(\SAut(\Xcal\times\A^1)\) implies in turn that $\Xcal\times\A^1$ is flexible.
This means that, at every point $q\in\Xcal\times\A^1$, the tangent
space $T_q(\Xcal\times\A^1)$ is spanned by tangent vectors to the
orbits of algebraic $\Ga$-actions. The same theorem also implies
that $\SAut(\Xcal\times\A^1)$ acts infinitely transitively, that is,
$m$-transitively for every $m\geq1$.

To prove this stronger transitivity statement, we only need, in addition to the
derivation \(E\) introduced above, the following three locally nilpotent
derivations of \(\C[x,y,z,t,u]\):
\[
\tau
=
\frac{\partial}{\partial u}
-x\frac{\partial}{\partial z}
+2u\frac{\partial}{\partial y},
\]
and
\[
\delta_z
=
x^2\frac{\partial}{\partial z}
-2(z+xu)\frac{\partial}{\partial y},
\qquad
\delta_t
=
x^2\frac{\partial}{\partial t}
-3t^2\frac{\partial}{\partial y}.
\]
They are the conjugates by \(\Theta\) of the translation
\(\partial/\partial u\) along the cylinder and of the following two standard 
locally nilpotent derivations of \(\C[\Xcal\times\A^1]\):
\[
x^2\frac{\partial}{\partial z}
-2z\frac{\partial}{\partial y}
\qquad\text{and}\qquad
x^2\frac{\partial}{\partial t}
-3t^2\frac{\partial}{\partial y}.
\]
 In particular, \(\tau\), \(\delta_z\), and \(\delta_t\) are locally
nilpotent and annihilate \(\widetilde P\).

Our main result can now be stated as follows.

\begin{theorem}
\label{thm:main}
Let
\[
G=
\left\langle
\exp(\lambda\partial)
\mathrel{}\middle|\mathrel{}
\lambda\in\C,\ 
\partial\in\{\tau,\delta_z,\delta_t,E\}
\right\rangle
\subset\Aut(\A^5).
\]
Then \(G\) preserves \(\Ycal\), and its induced action on \(\Ycal\)
is transitive. Consequently, \(\Xcal\times\A^1\) is flexible, and
\(\SAut(\Xcal\times\A^1)\) acts infinitely transitively.
\end{theorem}

The abundance of algebraic \(\Ga\)-actions on
\(\Xcal\times\A^1\) naturally raises the question whether some of
them arise from an action of a larger algebraic group. In this
direction, Freudenburg raised the following question, which remains
open \cite[Question~9.4]{Fre24}: does \(\Xcal\times\A^1\) admit a
nontrivial algebraic action of \(\mathrm{SL}_2(\C)\)?

The article is organized as follows.
In Section~\ref{sec:notation}, we introduce the notation associated
with the adapted embedding of the cylinder over the Russell cubic
threefold as the hypersurface \(\Ycal\subset\A^5\).
In Section~\ref{sec:DMJP}, we rewrite in these adapted coordinates
the locally nilpotent derivation constructed in \cite{DMJP10} that
does not annihilate \(x\).
In Section~\ref{sec:E}, we apply Bhatwadekar's technique to construct
the new derivation \(E\) and prove
Proposition~\ref{prop:new-LND}. Finally,
Section~\ref{sec:transitivity} is devoted to the proof of
Theorem~\ref{thm:main}.

\section{Notation}
\label{sec:notation}

 We work over the field \(\C\), and all
derivations are understood to be \(\C\)-derivations. 
If \(A\) is a domain and \(f\in A\setminus\{0\}\), we denote by
\(A_f=A[f^{-1}]\) the localization of \(A\) at \(f\).

We collect here the notation used in the introduction and throughout
the remainder of the article. We let 
\[\begin{aligned}
R &=\C[x,y,z,t,u],\\ 
\widetilde P &=x^2y+2xzu+z^2+x+t^3 \in R, \\ 
\Ycal &=V(\widetilde P)\subset\A^5, \\
B &=R/(\widetilde P)=\C[\Ycal].
\end{aligned}\]
By a slight abuse of
notation, we use the same letters \(x,y,z,t,u\) for the coordinate
functions on \(\A^5\) and for their residue classes in \(B\). 

Recall that the automorphism of \(\A^5\) defined by 
\[\Theta\colon(x,y,z,t,u)\mapsto(x,y+u^2,z-xu,t,u)\]restricts to an isomorphism
\[
\Theta\colon\Xcal\times\A^1\xrightarrow{\sim}\Ycal.
\]The plane
\[
H=\{x=z=t=0\}\subset\Ycal
\]
is the image under \(\Theta\) of
\(L\times\A^1\subset\Xcal\times\A^1\).

We shall repeatedly use the elements
\[
a=1+xy+2zu,\qquad c=zy-2au,\qquad m=zy+2au
\]
of \(R\). Direct computations give 
\[
\widetilde P=ax+z^2+t^3
\tag{\(\ast\)}
\label{eq:P-a}
\]
and
\[
zc+yt^3+a(a-1)=y\widetilde P.
\tag{\(\ast\ast\)}
\label{eq:fundamental-relation}
\]

\section{A first locally nilpotent derivation}
\label{sec:DMJP}

In \cite[Section~7]{DMJP10}, an explicit locally nilpotent derivation
of \(B\) which does not annihilate \(x\) was constructed. The
construction is based on the interpretation of the Russell cubic as
a one-parameter family of Danielewski surfaces parametrized by \(t\). After inverting \(t\),
an elementary locally nilpotent derivation is transported through an
isomorphism of cylinders, and a suitable power of \(t\) is then used
to clear the denominators. We present a slight modification of this
construction, in a form adapted to our notation.

Let
\[B'_t
=
\C[x,y,z,t^{\pm1},u]/(xy+z^2+t^3).
\]

\begin{lemma}
\label{lem:Phi}
The homomorphism \(\Phi\colon B'_t\to B_t\) defined by
\[
\Phi(x)=x,
\qquad
\Phi(y)=a,
\qquad
\Phi(z)=z,
\qquad
\Phi(t)=t,
\qquad
\Phi(u)=c
\]
is an isomorphism. Its inverse \(\Psi\colon B_t\to  B'_t\) is
given by
\[
\Psi(x)=x,
\qquad
\Psi(z)=z,
\qquad
\Psi(t)=t,
\]
and
\[
\Psi(y)=-\frac{y(y-1)+zu}{t^3},
\qquad
\Psi(u)=\frac{xu-z(y-1)}{2t^3}.
\]
\end{lemma}
We have included the explicit description of \(\Phi\) and \(\Psi\) to
make the construction of \(D\) transparent and to motivate the
formulas below. We omit the routine verification that these maps are
inverse to each other, since all the properties of \(D\) needed in the
sequel will be verified directly from those formulas. Notice in particular that \(a\) and \(c\) are the images under
\(\Phi\) of the coordinate functions \(y\) and \(u\), respectively. This explains their natural occurrence in the description of \(D\).

On \(B'_t\), consider the triangular locally nilpotent derivation
\[
\partial
=
y\frac{\partial}{\partial z}
-2z\frac{\partial}{\partial x}.
\]It annihilates \(xy+z^2+t^3\). Since
\(t\in\Ker(\partial)\), the derivation \(
D_t=\Phi\circ(t^3\partial)\circ\Psi
\) is locally nilpotent on \(B_t\). Although \(\Psi\) involves negative
powers of \(t\), a direct computation on the generators shows that the
factor \(t^3\) clears all the denominators. Thus \(D_t\) preserves
\(B\), and its restriction to \(B\) is given by the formulas in the
following proposition. 

\begin{proposition}\label{prop:D}
The following formulas define a locally nilpotent derivation \(D\) of
\(B\):
\[
\begin{aligned}
D(x)&=-2t^3z, & D(z)&=at^3, & D(y)&=-ac,\\
D(t)&=0,      & D(u)&=yt^3+\frac12a(a-1)
\end{aligned}
\]
Moreover, \(D(a)=D(c)=0\) and \(D(x)\ne0\).
\end{proposition}

\begin{proof}
Again, we omit the routine conjugation calculation and verify the
assertion directly from the displayed formulas. Let \(\widehat D\) be
the derivation of \(R\) defined by these formulas.

We compute directly
\[
\begin{aligned}
\widehat D(\widetilde P)
={}&\widehat D(x^2y+2xzu+z^2+x+t^3)\\
={}&2xy\widehat D(x)+x^2\widehat D(y)
 +2zu\widehat D(x)+2xu\widehat D(z)\\
&+2xz\widehat D(u)+2z\widehat D(z)
 +\widehat D(x)+3t^2\widehat D(t)\\
={}&-4xyzt^3-acx^2-4z^2ut^3+2axut^3\\
&+2xzyt^3+xza(a-1)+2azt^3-2zt^3\\
={}&ax\bigl(-xc+2ut^3+z(a-1)\bigr).
\end{aligned}
\]
Since \(c=zy-2au\), \(a-1=xy+2zu\), and by \eqref{eq:P-a}, we have
\[
-xc+2ut^3+z(a-1)
=2u(ax+z^2+t^3)=2u\widetilde P.
\]
Consequently, \(\widehat D(\widetilde P)=2aux\widetilde P\). Thus \(\widehat D\) induces a derivation \(D\) of \(B\).

Applying \(D\) to \eqref{eq:P-a} in \(B\), we obtain
\[
\begin{aligned}
0=D(\widetilde P)
 &=xD(a)+aD(x)+2zD(z)+3t^2D(t)\\
 &=xD(a).
\end{aligned}
\]
Since \(B\) is a domain and \(x\ne0\), it follows that \(D(a)=0\).

Using \(c=zy-2au\), we find
\[
\begin{aligned}
D(c)
&=D(z)y+zD(y)-2D(a)u-2aD(u)\\
&=ayt^3-acz-2a\left(yt^3+\frac12a(a-1)\right)\\
&=-a\bigl(zc+yt^3+a(a-1)\bigr)=0,
\end{aligned}
\]
where the last equality follows from
\eqref{eq:fundamental-relation}.

It remains to prove local nilpotence. Since \(a,c,t\in\Ker(D)\), we
have
\[
D^2(y)=D^2(z)=0.
\]
Furthermore,
\[
D^2(u)=-act^3,\qquad D^3(u)=0,
\]
and
\[
D^2(x)=-2at^6,\qquad D^3(x)=0.
\]
Thus \(D\) is nilpotent on a generating set of \(B\), and hence is
locally nilpotent. Finally, \(D(x)=-2t^3z\ne0\).
\end{proof}

\section{Construction of the new derivation}
\label{sec:E}

The purpose of this section is to prove
Proposition~\ref{prop:new-LND}. Before giving a direct proof, we
explain how the formulas defining \(E\) arise from the derivation
\(D\). As in the previous section, we omit the routine computations
leading to these formulas, since all the properties of \(E\) needed
below will be verified directly.

We begin by recalling the general procedure on which our construction
is based. We learned it from \cite[Section~7]{GP26}, where it is
referred to as Bhatwadekar's technique.

\begin{construction}[Bhatwadekar's technique]
\label{constr:Bhatwadekar}

Let \(\Delta\) be a locally nilpotent derivation of a domain \(A\),
and let \(w\in A\) be a local slice with \(\Delta(w)=h\in\Ker(\Delta)\setminus\{0\}\). After localizing at \(h\), the element \(w/h\) is a slice, and the
slice theorem gives \(A_h=(\Ker\Delta)_h[w/h]=(\Ker\Delta)_h[w]\).

Now, any locally nilpotent derivation \(\varepsilon\) of
\((\Ker\Delta)_h\) whose kernel contains \(h\) extends to a locally
nilpotent derivation of \(A_h\) by requiring that \(\varepsilon(w)=0\). Moreover, if there
exists an integer \(N\geq0\) such that \(h^N\varepsilon(A)\subset A\), then the restriction to \(A\) of the derivation \(h^N\varepsilon\) is locally
nilpotent, because \(h\in\Ker(\varepsilon)\). 
\end{construction}

We now apply Construction~\ref{constr:Bhatwadekar} to the derivation
\(D\). Recall that our goal is to construct an action moving a point
of \(H\) outside \(H\). In our situation, \(a\in\Ker(D)\) and, moreover,  \(a\) restricts to the constant function \(1\) on \(H\). This motivates the choice \(h=a\).

The element \(a\) does not, however, belong to the image of \(D\). To overcome this obstruction, we adjoin a
new variable \(v\) and extend \(D\) to \(B[v]\) by setting
\(D(v)=0\). Then, we consider the derivation
\[
\widetilde D=D+a\frac{\partial}{\partial v}.
\]
This derivation \(\widetilde D\) is locally nilpotent with \(\widetilde D(v)=a\), as needed.

According to Construction~\ref{constr:Bhatwadekar}, the next step is
to construct a locally nilpotent derivation of the localized kernel
\((\Ker\widetilde D)_a\) whose kernel contains \(a\).

To describe this localized kernel, set
\[
q=z-t^3v,
\qquad
r=y+cv.
\]
Since \(\widetilde D(z)=at^3\), \(\widetilde D(y)=-ac\), and
\(a,c,t\in\Ker(\widetilde D)\), we have
\(q,r\in\Ker(\widetilde D)\). A direct calculation shows that
\((\Ker\widetilde D)_a\) is generated as a \(\C\)-algebra by
\(a^{\pm1},c,t,q,r\), subject only to the relation \(qc+rt^3+a(a-1)=0\), which follows from \eqref{eq:fundamental-relation}. Thus
\[
(\Ker\widetilde D)_a=
\frac{\C[a^{\pm1},c,t,q,r]}
     {(qc+rt^3+a(a-1))}.
\]
We now define a derivation \(\varepsilon\) of
\(B[v]_a=(\Ker\widetilde D)_a[v]\) by
\[
\varepsilon(t)=c,
\qquad
\varepsilon(q)=-3rt^2,
\qquad
\varepsilon(a)=\varepsilon(c)=\varepsilon(r)=\varepsilon(v)=0.
\]
It is well defined, since \(\varepsilon\bigl(qc+rt^3+a(a-1)\bigr)
=-3rct^2+3rct^2=0\), and it is locally nilpotent because it is triangular.

Since \(a\in\Ker(\varepsilon)\), the derivation \(a\varepsilon\) is
locally nilpotent on \(B[v]_a\). Rewriting its values in terms of the generators of \(B[v]\) gives
\[
\begin{aligned}
(a\varepsilon)(x)&=3t^2m,
& (a\varepsilon)(y)&=0,
& (a\varepsilon)(z)&=-3ayt^2,\\
(a\varepsilon)(t)&=ac,
& (a\varepsilon)(u)&=-\frac32y^2t^2,
& (a\varepsilon)(v)&=0.
\end{aligned}
\]

All the expressions displayed above belong to \(B\), and hence to
\(B[v]\). Since \(x,y,z,t,u,v\) generate \(B[v]\), it follows that \(a\varepsilon(B[v])\subset B[v]\). Thus the denominator-clearing condition in
Construction~\ref{constr:Bhatwadekar} holds with \(N=1\), and
\(a\varepsilon\) restricts to a locally nilpotent derivation of
\(B[v]\). Furthermore, its values on \(x,y,z,t,u\) belong to \(B\),
so it preserves \(B\). Its restriction to \(B\) is precisely the
derivation \(E\) introduced in Proposition~\ref{prop:new-LND}.

\begin{proof}[Proof of Proposition~\ref{prop:new-LND}]
A direct computation gives
\[
E(a)=E(1+xy+2zu)
=3yt^2(m-zy-2au)=0.
\]
Consequently,
\[
E(c)=E(zy-2au)
=-3ay^2t^2+3ay^2t^2=0,
\]
and
\[
E(m)=E(zy+2au)=-3ay^2t^2-3ay^2t^2=-6ay^2t^2.
\]

Applying \(E\) to \eqref{eq:P-a}, we obtain
\[
E(\widetilde P)
=3at^2m-6ayzt^2+3act^2
=3at^2(m-2zy+c)=0.
\]

It remains to check local nilpotence on \(R\). Since
\(a,c,y\in\Ker(E)\) and \(E(t)=ac\), we have \(E^2(t)=0\) and
\(E^3(t^2)=0\). Moreover, \(E(z)\), \(E(u)\), and \(E(m)\) all
belong to \(\Ker(E)\,t^2\). It follows that
\[
E^4(z)=E^4(u)=E^4(m)=0.
\]
The Leibniz rule now gives \(E^6(t^2m)=0\). Hence
\(E^7(x)=E^6(3t^2m)=0\). Thus \(E\) is
nilpotent on the generators of \(R\), and is therefore locally
nilpotent.
\end{proof}

The element \(a\) used for the localization also plays a useful role
in the geometry of the resulting action. Since \(E(a)=0\), the
subvariety \(\Ycal\cap\{a=0\}\) is invariant under the \(E\)-action,
whose restriction to it takes a particularly simple form. We
conclude this section by recording this description, which will be
used in the proof of Theorem~\ref{thm:main}.

\begin{lemma}
\label{lem:E-orbits-a-zero}
Let \(p=(x,y,z,t,u)\in\Ycal\) satisfy \(a(p)=0\). Then, for every
\(\lambda\in\C\),
\[
\exp(\lambda E)(p)
=
\left(
x+3\lambda zyt^2,\,
y,\,
z,\,
t,\,
u-\frac32\lambda y^2t^2
\right).
\]
\end{lemma}

\begin{proof}
Since \(E(a)=0\), the ideal \((a)\subset R\) is \(E\)-stable.
Moreover,
\[
E(z)=-3ayt^2\in(a)
\qquad\text{and}\qquad
E(t)=ac\in(a).
\]
Using \(E(m)=-6ay^2t^2\), we obtain
\[
E^2(x)
=
E(3t^2m)
=
6actm-18ay^2t^4
\in(a),
\]
while
\[
E^2(u)
=
E\left(-\frac32y^2t^2\right)
=
-3acy^2t
\in(a).
\]
For every point \(p=(x,y,z,t,u)\in\Ycal\) with \(a(p)=0\) we therefore obtain
\[
E^k(x)(p)=E^k(u)(p)=0
\quad\text{for every }k\ge2,
\]
and
\[
E^k(z)(p)=E^k(t)(p)=0
\quad\text{for every }k\ge1.
\]

Finally, \(a(p)=0\) implies \(m(p)=zy\), and hence
\[
E(x)(p)=3zyt^2
\qquad\text{and}\qquad
E(u)(p)=-\frac32y^2t^2.
\]
The result now follows from the exponential formula.
\end{proof}

\section{Transitivity and flexibility}
\label{sec:transitivity}
Let \(G_0=
\left\langle
\exp(\lambda\tau),\,
\exp(\lambda\delta_z),\,
\exp(\lambda\delta_t)
\mathrel{}\middle|\mathrel{}\lambda\in\C
\right\rangle\) be the subgroup of \(G\) generated by the actions associated with
\(\tau\), \(\delta_z\), and \(\delta_t\). All these actions fix \(x\).
They are explicitly given by
\[
\exp(\lambda\tau)(x,y,z,t,u)
=
\bigl(
x,\,
y+2\lambda u+\lambda^2,\,
z-\lambda x,\,
t,\,
u+\lambda
\bigr),
\]
\[
\exp(\lambda\delta_z)(x,y,z,t,u)
=
\bigl(
x,\,
y-2\lambda(z+xu)-\lambda^2x^2,\,
z+\lambda x^2,\,
t,\,
u
\bigr),
\]
and
\[
\exp(\lambda\delta_t)(x,y,z,t,u)
=
\bigl(
x,\,
y-3\lambda t^2-3\lambda^2x^2t-\lambda^3x^4,\,
z,\,
t+\lambda x^2,\,
u
\bigr).
\]
We prove Theorem~\ref{thm:main} by showing that every point of
\(\Ycal\) can be sent by an element of \(G\) to
\[
q_1=(1,-1,1,-1,0).
\]
The argument proceeds in three steps. We first show that \(G_0\) acts
transitively on \(\Ycal\cap\{x=\xi\}\) for every \(\xi\ne0\). We
then apply Lemma~\ref{lem:E-orbits-a-zero} to show that every point
of \(\Ycal\setminus H\) can be sent to \(q_1\). Finally, we show
that every point of \(H\) can be moved into \(\Ycal\setminus H\).

\begin{proof}[Proof of Theorem~\ref{thm:main}]
\smallskip
\noindent\emph{Step 1: transitivity of \(G_0\) on the fibers
\(\Ycal\cap\{x=\xi\}\), for \(\xi\ne0\).}

Fix \(\xi\in\C^*\), and let
\(p=(\xi,y,z,t,u)\in\Ycal\). Using the \(\tau\)-action, we may first
arrange that \(u=0\). The \(\delta_z\)-action then allows us to make
\(z=0\), since it replaces \(z\) by \(z+\lambda\xi^2\), without
altering \(u\) or \(t\). Finally, the \(\delta_t\)-action allows us
to make \(t=0\), without altering \(u\) or \(z\). The equation
\(\widetilde P=0\) then forces \(y=-\xi^{-1}\). Thus \(p\) can be
sent to
\[
p_\xi=(\xi,-\xi^{-1},0,0,0).
\]
It follows that \(G_0\) acts transitively on
\(\Ycal\cap\{x=\xi\}\) for every \(\xi\in\C^*\).

\smallskip
\noindent\emph{Step 2: transitivity on \(\Ycal\setminus H\).}

Let \(p=(x,y,z,t,u)\in\Ycal\setminus H\). Suppose first that
\(x=\xi\ne0\). By Step~1, \(p\) can be sent by an element of \(G_0\)
to
\[
q_\xi=
\left(
\xi,-1,1,-1,\frac{\xi-1}{2}
\right).
\]
A direct calculation shows that \(q_\xi\in\Ycal\) and
\(a(q_\xi)=0\). By Lemma~\ref{lem:E-orbits-a-zero},
\[
\exp(\lambda E)(q_\xi)=q_{\xi-3\lambda}.
\]
Taking \(\lambda=(\xi-1)/3\), we send \(q_\xi\) to \(q_1\).

Suppose now that \(x=0\). Since \(p\notin H\), the relation
\(z^2+t^3=0\) implies that \(z\ne0\) and \(t\ne0\). Using the
\(\tau\)-action, followed by the \(\delta_z\)-action, we may send
\(p\) to
\[
r=
\left(
0,-1,z,t,-\frac{1}{2z}
\right).
\]
Indeed, the first action allows us to prescribe \(u\), while the
second allows us to prescribe \(y\) without altering \(z\), \(t\),
or \(u\). We have \(a(r)=0\), so
Lemma~\ref{lem:E-orbits-a-zero} gives
\[
\exp(\lambda E)(r)
=
\left(
-3\lambda zt^2,\,
-1,\,
z,\,
t,\,
-\frac{1}{2z}-\frac32\lambda t^2
\right).
\]
Taking
\[
\lambda=-\frac{1}{3zt^2},
\]
we obtain
\[
\exp(\lambda E)(r)=(1,-1,z,t,0).
\]
This point belongs to \(\Ycal\cap\{x=1\}\), so Step~1 allows us to
send it to \(q_1\). Thus every point of \(\Ycal\setminus H\) can be
sent to \(q_1\), and \(G\) acts transitively on
\(\Ycal\setminus H\).

\smallskip
\noindent\emph{Step 3: leaving \(H\).}

Let \(p=(0,y,0,0,u)\in H\). Using the \(\tau\)-action, we may send
\(p\) to a point of \(H\) whose \(u\)-coordinate is \(1\). At this
point,
\[
a=1
\qquad\text{and}\qquad
c=-2,
\]
so
\[
E(t)=ac=-2\ne0.
\]
The restriction of \(t\) to the corresponding \(E\)-orbit is
therefore nonconstant. Hence this orbit contains a point outside
\(H\), which can be sent to \(q_1\) by Step~2.

Thus every point of \(\Ycal\) can be sent to \(q_1\), and \(G\) acts
transitively on \(\Ycal\). The remaining assertions follow from
\cite[Theorem~0.1]{AFKKZ13}, as explained in the introduction.
\end{proof}

\subsection*{AI disclosure}

Part of this work was developed through a series of chat-based discussions with ChatGPT, using GPT-5.6 Sol model. During these discussions,
the model drew the author's attention to Bhatwadekar's technique and
assisted with the computations leading to the construction of the
locally nilpotent derivation \(E\). 

\end{document}